\documentclass[12pt,a4paper]{article}
\usepackage{graphicx}
\usepackage{amsfonts}
\usepackage{amsmath}
\usepackage{amssymb}
\usepackage{amsthm}
\usepackage{color}
\definecolor{darkred}{RGB}{139,0,0}
\definecolor{darkgreen}{RGB}{0,100,0}
\definecolor{darkmagenta}{RGB}{139,0,139}
\definecolor{darkpurple}{RGB}{110,0,180}
\definecolor{darkblue}{RGB}{40,0,200}
\definecolor{darkorange}{RGB}{255,140,0}
\definecolor{darkpink}{RGB}{255,20,147}
\usepackage[ansinew]{inputenc}
\usepackage{latexsym}
\usepackage{tabularx}
\usepackage{graphicx}
\usepackage{latexsym,amsfonts,amsmath,amssymb,mathrsfs}
\usepackage[active]{srcltx}
\newtheorem{theorem}{Theorem}
\newtheorem{corollary}{Corollary}
\newtheorem{lemma}{Lemma}
\newtheorem{example}{Example}
\newtheorem{proposition}{Proposition}

\newtheorem{remark}{Remark}

\newtheorem{algorithm}{Algorithm}
\newcommand{\rd}{\,\mathrm{d}}
\newcommand{\bsx}{\boldsymbol{x}}
\newcommand{\bsp}{\boldsymbol{p}}

\begin{document}
	
	\title{A Discrepancy Bound for Deterministic Acceptance-Rejection Samplers Beyond $N^{-1/2}$ in Dimension 1}

	\author{Houying Zhu  \and  Josef Dick}

	
	\date{School of Mathematics and Statistics, \\ The University of New South Wales, Australia}

	\maketitle
	
	\begin{abstract}
		In this paper we consider an acceptance-rejection (AR) sampler based on deterministic driver sequences.  We prove that the discrepancy of an $N$ element sample set generated in this way is bounded by $\mathcal{O} (N^{-2/3}\log N)$, provided that the target density is twice continuously differentiable with non-vanishing curvature and the AR sampler uses the  driver sequence  $$\mathcal{K}_M= \{( j \alpha, j \beta ) ~~ mod~~1 \mid j = 1,\ldots,M\}, $$ where $\alpha,\beta$ are real algebraic numbers such that $1,\alpha,\beta$ is a basis of a number field over $\mathbb{Q}$ of degree $3$. For the driver sequence $$\mathcal{F}_k= \{ ({j}/{F_k}, \{{jF_{k-1}}/{F_k}\} ) \mid j=1,\ldots, F_k\},$$ where $F_k$ is the $k$-th Fibonacci number and $\{x\}=x-\lfloor x \rfloor$ is the fractional part of a non-negative real number $x$, we can remove the $\log$  factor to improve the convergence rate to $\mathcal{O}(N^{-2/3})$,  where again $N$ is the number of samples we accepted.
		
		We also  introduce a criterion for measuring the goodness of driver sequences.  The proposed approach is numerically tested  by calculating the star-discrepancy of samples generated for some target densities using $\mathcal{K}_M$ and $\mathcal{F}_k$ as driver sequences. These results confirm that achieving a convergence rate beyond $N^{-1/2}$ is possible in practice using $\mathcal{K}_M$ and $\mathcal{F}_k$ as driver sequences in the acceptance-rejection sampler. 
		
	{\bf{Keywords}}: Acceptance-rejection sampler,  discrepancy,  Fibonacci lattice points, integration error.
	\end{abstract}

	\section{Introduction}
	Many applications involve sampling from a non-uniform distribution where  direct simulation is difficult  or even
	impossible \cite{Devroye1986,HLD2004}. One possible way to obtain samples from such distributions is to choose
	a proposal density from which we can sample  with low cost and then take part of the initial points as samples of the target density under some rules. The implementation of this idea is known as acceptance-rejection technique. We are interested  in the properties of samples generated in this way, especially the discrepancy of those points.
	By a generalization of the definition of the classical discrepancy, the discrepancy with respect to a distribution is defined in the following way. Let  $P_N=\{\bsp_j \mid j=1,\ldots,N\} \subset [0,1]^{s-1}$. Let $\psi:[0,1]^{s-1}\to \mathbb{R}_+$, where $\mathbb{R}_+$ denotes  the set of non-negative real numbers. Then the star-discrepancy of the point set $P_N$ with respect to $\psi$ is given by
	\begin{equation}
	D_{N,\psi}^*(P_N)=\sup_{\boldsymbol{t}\in[0,1]^{s-1}} \left|\displaystyle \frac{1}{N}\sum_{j=1}^{N} 1_{[\boldsymbol{0},\boldsymbol{t})}({\boldsymbol{p}_j})
	-\displaystyle\frac{1}{C}\int_{[\boldsymbol{0},\boldsymbol{t})}\psi({\boldsymbol{z}})d\boldsymbol{z}\right|,
	\end{equation}\label{defStarDis}
	where $$C=\int_{[0,1]^{s-1}}\psi(\boldsymbol{z})d\boldsymbol{z}>0$$ is the normalizing constant, $[\boldsymbol{0},\boldsymbol{t})=\prod_{i=1}^{s-1}[0,t_i)$, and $1_{[\boldsymbol{0},\boldsymbol{t})}$ is the indicator function of $[\boldsymbol{0},\boldsymbol{t})$, i.e., it is $1$ if $\bsp_j \in [\boldsymbol{0},\boldsymbol{t})$ and $0$ otherwise.  The special case where $\psi=1$  is denoted by $D_N^*(P_N)$.
	
	The aim of this paper is to devise methods to generate sample sets with low-discrepancy  with respect to a given unnormalized target distribution. This is in general a difficult problem and not much is known on how low-discrepancy point sets can be generated. The existence of low-discrepancy point sets with respect to an arbitrary non-uniform measure in $[0,1]^s$ has been discussed by Aistleitner and Dick in \cite{AD2014}. They proved that there exists an $N$ point set in $[0,1]^s$ whose star-discrepancy associated with a normalized Borel measure has the order of magnitude ${(\log N)^{\frac{3s+1}{2}}}/{N}$. Since the result in \cite{AD2014} is based on probabilistic arguments, it is not known how to explicitly construct good point sets. In general, it appears to be a very difficult problem to explicitly construct point sets achieving this rate of convergence. Even how to explicitly construct point sets which achieve a convergence rate beyond $N^{-1/2}$ is not known in most cases. In this paper we show that in dimension 
	$1$ it is possible to explicitly construct point sets with discrepancy of order $N^{-2/3}$ for the more general case where the target density is only known up to a constant but satisfies some smoothness conditions.

	The star-discrepancy of samples appears in the upper bound on the integration error in the Koksma-Hlawka inequality \cite{AD2015,BCGT2015}
	\begin{eqnarray*}
		\Big| \frac{1}{N}\sum_{j=1}^{N} f(\bsp_j)-\frac{1}{C}\int_{[0,1]^{s-1}} f(\bsx) \psi(\bsx) \rd \bsx\Big|
		\le V_{HK}(f) D_{N,\psi}^*(P_N),
	\end{eqnarray*} 
	where $V_{HK}(f) $ is the Hardy-Krause variation of the function, see \cite{AD2015,BCGT2015} and Section~1.2 below for details. 	This makes the star-discrepancy an important quality criterion for sample sets.
	
	If the goal is to use the low-discrepancy point sets for numerical integration, then sometimes alternative methods can also be used. For instance,
	\begin{equation*}\centering
	{\frac{1}{C}\int_{[0,1]^{s-1}} f(\bsx)\psi(\bsx) \rd \bsx}
	\end{equation*}
	can be approximated by
	\begin{equation*}\centering
	{\Big(\sum\limits_{j=1}^{N} w_j \psi(\bsp_j)\Big)^{-1}  \sum\limits_{j=1}^{N} w_j f(\bsp_j) \psi(\bsp_j)},
	\end{equation*}
	where we used some quadrature rule with weights $w_j$ and quadrature points $\bsp_j$ for $j = 1, \ldots, N$. See for instance \cite{BG2004,C1997,C2003,DP2010} for more background on suitable quadrature rules.
	
	However, this method may not always yield good results. For instance, if $\int_{[0,1]^{s-1}} \psi(\bsx) \rd \bsx$ is small, then even small errors in approximating $\int_{[0,1]^{s-1}} \psi(\bsx) \rd \bsx$ by $\frac{1}{N} \sum\limits_{j=1}^{N} \psi(\bsp_j)$ may yield large errors overall. In this case our method based on acceptance-rejection provides a good alternative which still yields small integration errors. Further, this method does not yield point sets with low discrepancy with respect to the unnormalized target density, which is the main aim of this paper.

	In this paper we focus on the  acceptance-rejection method which can be used to obtain samples with distribution $\psi$  through a rather simple procedure as indicated below. The acceptance-rejection algorithm based on random inputs works as follows. 
	
	\begin{algorithm}\label{Alg:AR} 
		For a given a target density $\psi: [0,1]^{s-1}\to{\mathbb{R}}_{+}$, choose a proposal density $H: [0,1]^{s-1}\to \mathbb{R}_+$ such that there exists a constant $L<\infty$ with  $\psi(\boldsymbol{x}) < LH(\boldsymbol{x})$ for all $\boldsymbol{x}$ in $[0,1]^{s-1}$. Then the acceptance-rejection algorithm is given by
		\begin{enumerate}
			\item [i)~] Draw $X\thicksim H$ and $u\thicksim U([0,1])$.
			\item [ii)~]  Accept $Y=X$ as a sample of $\psi$ if $u\leq \frac{\psi({X})}{LH({X})}$,
			otherwise go back to the first step.
		\end{enumerate}
	\end{algorithm}

	To improve the performance of this algorithm, there are at least two aspects which  can be studied, namely
	choosing a better proposal density $H$ \cite{Devroye1986,HLD2004} or proposing good initial samples with respect to a well chosen $H$. We focus on the latter point in this paper. Note that we always choose the uniform distribution as our proposal distribution.
	
	The acceptance-rejection algorithm based on a deterministic driver sequence works as follows.
	\begin{algorithm}\label{DAR} For a given a target density $\psi:[0,1]^{s-1}\to{\mathbb{R}}_{+}$, assume that there exists a constant $L<\infty$ such that  $\psi(\boldsymbol{x}) < L$ for all $\boldsymbol{x}\in [0,1]^{s-1}$. Let the uniform distribution be the proposal density. Let $A=\{\boldsymbol{x}\in[0,1]^s \mid \psi(x_1, \ldots, x_{s-1})\geq L x_s\}$.
		\begin{itemize}
			\item [i)~] Generate a uniformly distributed point set $T_M=\{\bsx_j \mid j= 1,\ldots, M \}$ in $[0,1]^s$.
			\item [ii)~]Use the acceptance-rejection method for the points $T_M$ with respect to the density $\psi$, i.e., we accept the point $\boldsymbol{x}_j$ if $\boldsymbol{x}_j \in A$, otherwise reject. Let $\widetilde{Q}_N(T_M; \psi) = A\cap T_M$ be the sample set which we accepted.
			\item [iii)~]  Project the points in $\widetilde{Q}_N(T_M; \psi)$ onto  the first $s-1$ coordinates and denote the resulting point set by $Q_N(T_M; \psi)=\{\boldsymbol{y}_j \mid j = 1, \ldots, N \}\subset [0, 1]^{s-1}$.
			\item [iv)~] Return the point set $Q_N(T_M; \psi)$.
		\end{itemize}
	\end{algorithm}
	
	Via the definition of the acceptance-rejection algorithm, we can view $Q_N$ as a function of point sets $T_M$ and unnormalized target densities $\psi$, which maps to point sets of size $N$ (where $N$ is itself a function of the particular inputs). Note that the function $Q_N$ (and in particular also $N$) also depends on the choice of $L$, since $L$ is not uniquely defined by $\psi$. We mainly suppress the dependence on $L$ in the following.
	
	We restrict our investigation to the case where the density function is defined in $[0,1]$ in this paper  (which can be generalized to any bounded interval), since in this case we can prove bounds on the star-discrepancy beyond  the order $N^{-1/2}$. Generalizing our results to dimension $s>1$ would be of great interest, but this would require an extension of results on the discrepancy with respect to  convex sets with smooth boundary and non-vanishing 	curvature, which is currently not known. See \cite{BCGT2015} for more details.
	
	\subsection{Discrepancy bounds}
	In \cite{ZD2014} we proposed a deterministic acceptance-rejection sampler using low-discrepancy point sets as driver sequences. Therein we proved that the star-discrepancy is of order  $N^{-1/s}$ for density functions  defined in the $(s-1)$-dimensional cube, using  $(t,m,s)$-nets or $(t,s)$-sequences as driver sequences. However, numerical results suggested a much better rate of convergence. Additionally, we proved  a lower bound on the star-discrepancy with respect to a concave density function. The lower bound suggests a convergence rate of the form $N^{-\frac{2}{s+1}}$ for  density functions defined in $[0,1]^{s-1}$.
	We recall the lower bound here.
	
	\begin{proposition}\label{B:lower} Let $T_M$ be an arbitrary  point set  in $[0,1]^s$. Then there exists a concave density function $\psi: [0,1]^{s-1} \to [0, 1]$ such that, for $N$ samples $Q_N(T_M; \psi)$ generated by the acceptance-rejection algorithm with respect to $T_M$ and $\psi$, we have
		\begin{equation*}
		D_{N,\psi}^*(Q_N(T_M; \psi))\ge c_sN^{-\frac{2}{s+1}},
		\end{equation*}
		where $c_s >0$ is independent of $N$ and $T_M$ but depends on $s$.
	\end{proposition}

	It is natural to ask whether the above bound is achievable, i.e., can we construct a driver sequence $T_M$ such that the discrepancy of the point set $Q_N(T_M; \psi)$ achieves a convergence rate of (almost) $N^{-2/3}$ in dimension 1  for a class of target densities $\psi$ (note that dimension $1$ corresponds to $s=2$ in Proposition~\ref{B:lower}).   
	Here we present two types of constructions of driver sequences $T_M$ for which this property holds for the class of twice continuously differentiable target densities with non-vanishing curvature. The first one is shown in Theorem~\ref{Th:2thirds} below. It uses the notions of non-vanishing curvature and number fields of degree $3$, which we explain in the following.
	
	The curvature of a twice continuously differentiable plane curve $\gamma(t) = (x(t), y(t))$ used in this paper is  defined as $$\kappa(t) = \frac{x'(t) y''(t) - y'(t) x''(t)}{( (x'(t) )^2 + (y'(t) )^2)^{3/2}},$$ where the parameterization is such that $(x'(t))^2 + (y'(t))^2 \neq 0$ for all $t$ in the domain. If the curve is given by a function $\gamma(t) = (t, f(t))$, then this reduces to $$\kappa(t) = \frac{f''(t)}{(1 + (f'(t))^2)^{3/2}}.$$ Recall that if $f$ is concave, then $\kappa \le 0$ and if $f$ is convex then $\kappa \ge 0$. By the assumption that the curve is twice continuously differentiable we have that $\kappa$ is continuous. The assumption that the curvature is non-vanishing implies therefore that $|\kappa(t)| \ge c > 0$ for all $t$ in the domain, and by the continuity of $t$, $\kappa$ is either positive or negative everywhere. In particular, if the curve is given by a function, this means that the function is either strictly concave or 
	strictly convex everywhere.
	
	For the construction of a suitable driver sequence we use algebraic number fields over the set of rational numbers $\mathbb{Q}$. An algebraic number field over $\mathbb{Q}$ is a finite degree field extension of the field $\mathbb{Q}$ of rational numbers and its dimension as a vector space over $\mathbb{Q}$ is called the degree of the number field. For instance, the set $\mathbb{Q}(\xi, \xi^2) = \{a + b \xi + c \xi^2 \mid a, b, c \in \mathbb{Q}\}$, where $\xi$ is a real root of a third degree irreducible polynomial over $\mathbb{Z}$, is a (real) number field of degree $3$. In this case, $\{1, \xi, \xi^2\}$ is a basis of the number field.
	
	\begin{theorem}\label{Th:2thirds}
		Let an unnormalized density function $\psi:[0,1]\to \mathbb{R}_+$  be twice continuously differentiable and having non-vanishing curvature.
		Assume that there exists a constant $L<\infty$ such that $\psi(x) \leq L$ for  all $x\in[0,1]$. Let \begin{eqnarray*}
			\mathcal{K}_M= \{ ( j \alpha, j\beta ) ~~ mod~~1 \mid j=1,\ldots,M \},
		\end{eqnarray*}
		where $\alpha,\beta$ are real algebraic numbers such that $1,\alpha,\beta$ is a basis of a number field over  $\mathbb{Q}$ of degree $3$. Then the discrepancy of  $Q_N (\mathcal{K}_M; \psi) \subset [0,1]$ generated by the  acceptance-rejection sampler using $\mathcal{K}_M$ as  driver sequence satisfies
		\begin{eqnarray*}
			D^*_{N,\psi}( Q_N (\mathcal{K}_M; \psi))\le C_{\psi, L}N^{-2/3}\log N,
		\end{eqnarray*}
		where  $C_{\psi, L}$ is a constant depending only on the target density $\psi$ and the constant $L$.
	\end{theorem}
	
	The proof  of this result  is presented  in Section~\ref{sec:PT1}. We also prove the following result which improves the previous bound by a factor of $\log N$. Before we can state this result, we introduce the following notation. Let $F_k$ denote the $k$-th Fibonacci number  given by $$F_1=F_2=1, F_k=F_{k-1}+F_{k-2} \mbox{ for } k\ge 3.$$ Let $$\{x\}=x-\lfloor x \rfloor$$ denote the fractional part of the non-negative real number $x$.
	\begin{theorem}\label{Th:lattice}
		Let an unnormalized density function $\psi:[0,1]\to \mathbb{R}_+$  be twice continuously differentiable and having non-vanishing curvature.  Assume that there exists a constant $L<\infty$ such that $\psi(x) \leq L$ for  all $x\in[0,1]$.
		
		Let
		\begin{eqnarray*}
			\mathcal{F}_k= \left\{ \Big(\frac{j}{F_k}, \big\{\frac{jF_{k-1}}{F_k}\big\} \Big) \mid j=1,\ldots,  F_k\right\}.
		\end{eqnarray*}
		Then the  discrepancy of $Q_N (\mathcal{F}_k; \psi) \subset [0,1]$ generated by the acceptance-rejection sampler using $\mathcal{F}_k$ as driver sequence satisfies
		\begin{eqnarray*}
			D^*_{N,\psi}( Q_N(\mathcal{F}_k; \psi))\le C'_{\psi, L} N^{-2/3},
		\end{eqnarray*} where $C'_{\psi, L}$ is a constant depending only on the target density $\psi$ and the constant $L$.
	\end{theorem}
	
	
	To prove this result we first introduce a  quality criterion for driver sequences in Section~\ref{sec_quality} (see Equation~\eqref{CriterionR}) and then prove a bound on this criterion for the set $\mathcal{F}_k$ in Section~\ref{sec:PT2}.

	\subsection{Integration error}
	
	In \cite{AD2015}, Aistleitner and Dick proved a generalized Koksma-Hlawka inequality for non-uniform measures which states that for any function having bounded variation in the sense of Hardy and Krause (abbreviated as $V_{HK}$), the integration error can be bounded by a product of  the variation of the integrand function times the discrepancy of the quadrature points. With the definition of $V_{HK}$ given in \cite[Section~2]{AD2015}, the following  inequality holds.

	\begin{proposition}\label{Prop:GeneralKHIneq}
		Let $f$ be a measurable function on $[0,1]^{s-1}$ which has bounded variation in the sense of Hardy and Krause,  $V_{HK}(f) <\infty$. Let $\mu$ be a normalized Borel measure on $[0,1]^{s-1}$, and let $P_N=\{\bsp_j \mid j = 1, \ldots, N \}$ be a point set in $[0,1]^{s-1}$. Then 
		\begin{equation*}
		\Big| \frac{1}{N}\sum_{j=1}^{N} f(\bsp_j)-\int_{[0,1]^{s-1}} f(\bsx) \rd \mu(\bsx)\Big|\le V_{HK}(f) D_{N,\mu}^*( P_N).
		\end{equation*}
	\end{proposition}
	In the following we use this result for $s=2$. In this case, if the function  $f$ is absolutely continuous,
	then the variation in the sense of Hardy and Krause
	can be written as $$V_{HK} (f)=\int_0^1 |f'(x)|\rd x.$$
	Theorem~\ref{Th:2thirds}, Theorem~\ref{Th:lattice} and Proposition~\ref{Prop:GeneralKHIneq} imply the following result.
	
	\begin{corollary}\label{Corr:integErr} Let $f:[0,1]\to \mathbb{R}$ have bounded variation $V_{HK}(f) <\infty$. Let $\psi$  be non-negative, twice continuously differentiable and having non-vanishing curvature. Then we have
		\begin{align*}
		&\Big| \frac{1}{N}\sum_{x \in R_N} f(x)-\frac{1}{C}\int_{0}^1 f(x)\psi(x) \rd x\Big|\\
		\le&
		\left\{\begin{array}{ll}
		C_{\psi, L} V_{HK}(f)  N^{-2/3}\log N, &\quad  \mbox{for } R_N = Q_N(\mathcal{K}_M; \psi),\\[0.25cm]
		C'_{\psi, L} V_{HK}(f)  N^{-2/3}, &\quad \mbox{for } R_N = Q_N(\mathcal{F}_k; \psi),\\[0.25cm]
		\end{array}\right.
		\end{align*}
		where $ N$ is the number of points we accepted, $C=\int_{[0,1]} \psi(x) \rd x>0$ and where $C_{\psi, L}$  and $C'_{\psi, L}$ are  constants depending only on the target density $\psi$ and the constant $L$.
	\end{corollary}

	\section{Numerical experiments}
	
	To demonstrate the performance of  the deterministic acceptance-rejection samplers,
	we consider two density functions defined on $[0,1]$ and calculate the star-discrepancy of the
	samples generated by the proposed methods.  For comparison purpose, the convergence rate  for the original algorithm  using {random} points and regular grids as driver sequence are also presented. Note that numerical results are presented in a log-log scale in the figures. 
	\begin{example} Let the target density $\psi:[0,1]\to \mathbb{R}_+$ be given by
		\begin{equation*}
		\psi(x)= \frac{3}{16} \left(4 \sin\left( \frac{\pi x }{ 2} \right) - x^{5/2} - x^2 \right).
		\end{equation*}
	\end{example}

	This density function satisfies all the conditions in our  theory since it is twice continuously differentiable and strictly concave. The numerical results shown in Figure~\ref{Fig:E1} suggest an empirical convergence rate of approximately $N^{-0.75}$ for samples of $\psi$ obtained by the  deterministic acceptance-rejection sampler using the driver sequence
	\begin{equation*}
	\mathcal{K}_M= \{( j \alpha, j \beta ) ~~ mod~~1 \mid j=1,\ldots,M\}.
	\end{equation*}
	In the test we choose the real root of the polynomial $x^3 + 2x + 2$. Eisenstein's criterion implies that this polynomial is irreducible over $\mathbb{Z}$. The root $\xi$ is approximated by $-0.770916997059248$ and we set $\alpha = \xi$ and $\beta = \xi^2$.
	
	
	Similarly, using the Fibonacci lattice point set
	\begin{equation*}
	\mathcal{F}_k= \{({j}/{F_k}, \{{jF_{k-1}}/{F_k}\} ) \mid j=1,\ldots, F_k\},
	\end{equation*}
	the numerical experiments show a convergence rate of approximately $N^{-0.8}$. The original acceptance-rejection sampler in the random setting produces samples whose star-discrepancy converges at roughly $N^{-1/2}$. A similar result is observed for the regular grid $G_M$ given by
	\begin{equation*}
	G_M=\left\{\Big(\frac{j}{\lfloor \sqrt{M}\rfloor}, \frac{m}{\lfloor\sqrt{M}\rfloor}\Big) \mid \, j,m=1,\ldots,\lfloor\sqrt{M}\rfloor\right\}. 
	\end{equation*}
	It is worth noticing that Fibonacci lattice points always provided the smallest value of the discrepancy. The acceptance rate is roughly $69\%$ for the first example.
	
	\begin{figure}[!htp]
		\begin{center}
			\includegraphics[trim=0.2cm 0.05cm 0.03cm 0.2cm, clip=true,width=0.8\textwidth]{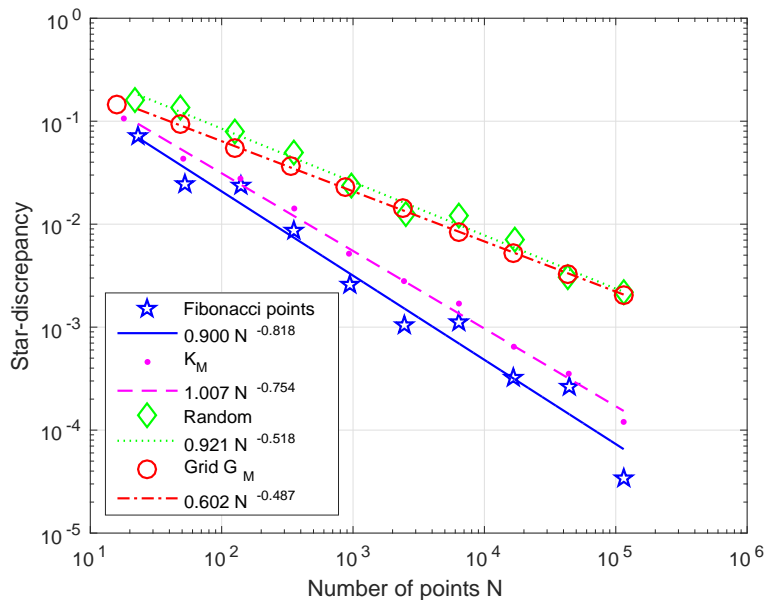}\\
			\begin{minipage}{0.8\textwidth}\caption{{\footnotesize  Convergence order of the star-discrepancy with respect to  different driver sequences of Example~1.}}\label{Fig:E1}
			\end{minipage}
		\end{center}
	\end{figure}
	
	\begin{example}
		Consider the twice continuously differentiable and strictly convex target density function
		\begin{center}
			\begin{equation*}
			\psi(x)=
			\left\{\begin{array}{ll}
			-\frac{1}{2} x^4-\frac{1}{6}x^2+\frac{107}{108},\quad & x\in [0, \frac{1}{3}), \\ [0.3cm]
			- \frac{3}{4}x^4-\frac{2}{27}x+1,\quad & x\in [\frac{1}{3}, 1].
			\end{array} \right.
			\end{equation*}
		\end{center}
	\end{example}
	
	We again observe  much better results with deterministic driver sequences, $\mathcal{F}_k$  and $\mathcal{K}_M$, compared with pseudo-random points  and regular grids as shown in Figure~\ref{Fig:E2}. As observed in the first example, a Fibonacci lattice point set $\mathcal{F}_k$ yields a slightly better numerical result compared to the point set $\mathcal{K}_M$ in this experiment. The acceptance rate is around $80\%$ for Example~2.
	\begin{figure}[!htp]
		\begin{center}
			\includegraphics[trim=0.2cm 0.05cm 0.03cm 0.2cm, clip=true ,width=0.8\textwidth, angle=360]{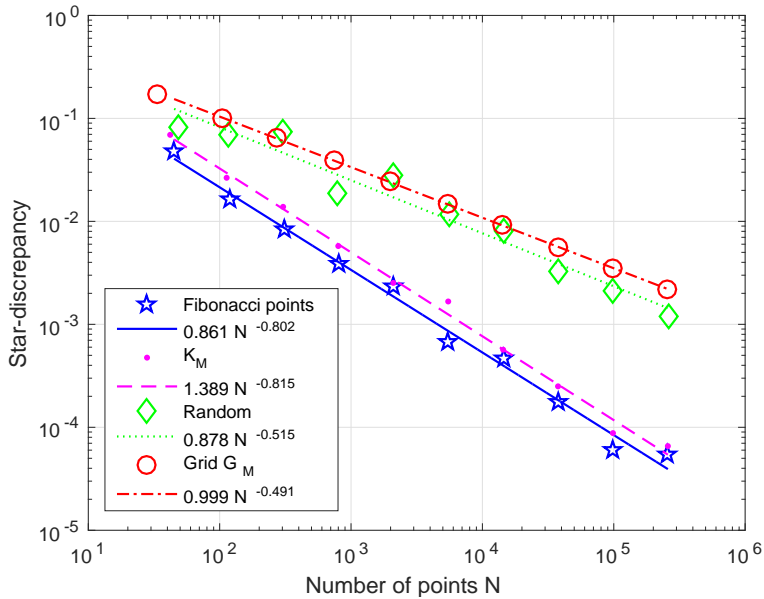}\\
			\begin{minipage}{0.8\textwidth}\caption{{\footnotesize Convergence order of the star-discrepancy with respect to  different driver sequences of Example~2.}}\label{Fig:E2}
			\end{minipage}
		\end{center}
	\end{figure}

	\section{Proof of Theorem~\ref{Th:2thirds}}\label{sec:PT1}
	The proof of Theorem~\ref{Th:2thirds} is motivated by a recent paper due to
	Brandolini et  al. \cite{BCGT2015}.  Therein they proved an upper bound for the following discrepancy associated with a convex domain with smooth boundary. We recall the main results pertaining to our paper here.

	Let $\Omega$ be a bounded convex domain in $\mathbb{R}^2$ such that the boundary curve is twice continuously differentiable and having non-vanishing curvature.
	For $\boldsymbol{t}=(t_1,t_2)\in (0,1)^2$ and  any $\bsx \in \mathbb{R}^2$,
	let
	\begin{equation*}
	I(\boldsymbol{t},\bsx)= \bigcup_{\boldsymbol{m}\in \mathbb{Z}^2} ([0,t_1]\times [0,t_2]+\bsx+\boldsymbol{m}).
	\end{equation*}
	Consider the following discrepancy defined with respect to the set  $\Omega$ and a point set $P_N=\{\bsp_j \mid j=1, \ldots, N\}$ in $\mathbb{R}^2$,
	\begin{eqnarray}\label{Dis:omega}
	\widetilde{D}_N^*(P_N,\Omega)=\sup_{\substack{\boldsymbol{t}\in [0,1]^2\\ \bsx\in \mathbb{R}^2}} \Big| \frac{1}{N}\sum_{j=1}^{N} \sum_{\boldsymbol{m}\in \mathbb{Z}^2} 1_{I(\boldsymbol{t},\bsx)\cap \Omega} (\bsx_j +\boldsymbol{m}) 
	-\lambda(I(\boldsymbol{t},\bsx)\cap \Omega)\Big|,
	\end{eqnarray}where $\lambda$ denotes the Lebesgue measure.
	
	The following result is \cite[Theorem~2]{BCGT2015}.
	\begin{proposition}\label{Prop:Omega} Suppose that $\Omega$ is a convex domain in $\mathbb{R}^2$ such that the boundary curve is twice continuously differentiable and having non-vanishing  curvature. Let $1,\alpha,\beta$ be a basis of a number field over $\mathbb{Q}$ of degree $3$. Let $$\mathcal{K}_N= \{ \bsx_j = ( j \alpha, j \beta ) \mid j = 1,\ldots,N\}. $$ For the  discrepancy defined in Equation~\eqref{Dis:omega}, we have
		$$ \widetilde{D}_N^*(\mathcal{K}_N,\Omega) \le cN^{-2/3} \log N,$$
		where the constant $c$ depends on the minimum and  maximum of the curvature of the boundary of $\Omega$ and the length of the boundary, and on the numbers
		$\alpha$ and $\beta$.
	\end{proposition}
	
	The proof of \cite[Theorem~2]{BCGT2015} actually shows that a slightly more general statement holds, which we describe in the following. 
	
	For given $\boldsymbol{t}$ and $\boldsymbol{x}$, the set $I(\boldsymbol{t}, \boldsymbol{x})$ is the union of infinitely many rectangles of the form $[0, t_1] \times [0, t_2] + \boldsymbol{x} + \boldsymbol{m}$, where $\boldsymbol{m} \in \mathbb{Z}^2$. Let $K_1, \ldots, K_q$ denote all those rectangles which have non-empty intersection with $\Omega$, i.e., $K_r = [0, t_1] \times [0, t_2] + \boldsymbol{x} + \boldsymbol{m}_r$ for suitable choices of $\boldsymbol{m}_r \in \mathbb{Z}^2$ with $K_r \cap \Omega \neq \emptyset$. Then
	\begin{align}\label{bound_BCGT}
	& \frac{1}{N}\sum_{j=1}^{N} \sum_{\boldsymbol{m}\in \mathbb{Z}^2} 1_{I(\boldsymbol{t},\bsx)\cap \Omega} (\bsx_j +\boldsymbol{m}) -\lambda(I(\boldsymbol{t},\bsx)\cap \Omega) \nonumber \\ = & \sum_{r=1}^q \left(\frac{1}{N} \sum_{j=1}^N \sum_{\boldsymbol{m} \in \mathbb{Z}^2}  1_{K_r \cap \Omega}(\bsx_j + \boldsymbol{m} )  - \lambda(K_r \cap \Omega) \right).
	\end{align}
	In \cite[p. 10]{BCGT2015} the authors state that they prove their result by showing the upper bound for a single piece $K_r$, i.e. they show the bound 
	\begin{align}\label{bound_BCGT2}
	 \sup_{\substack{\boldsymbol{t}\in [0,1]^2\\ \bsx\in \mathbb{R}^2}} \left|\frac{1}{N} \sum_{j=1}^N \sum_{\boldsymbol{m} \in \mathbb{Z}^2}  1_{K_r \cap \Omega}(\bsx_j + \boldsymbol{m} )  - \lambda(K_r \cap \Omega) \right| \le  c'_s N^{-2/3} \log N.
	\end{align}
	The bound on $\widetilde{D}_N(\mathcal{K}, \Omega)$ then follows by the triangle inequality. We use \eqref{bound_BCGT2} rather then Proposition~\ref{Prop:Omega} in the following.
	
	Note that we are only interested in sets $K_r$ which are contained in the unit square, i.e. $K_r \subset [0,1]^2$. In this case we have
	\begin{equation*}
	\frac{1}{N}\sum_{j=1}^{N} \sum_{\boldsymbol{m} \in \mathbb{Z}^2} 1_{K_r \cap \Omega} (\bsx_j + \boldsymbol{m}) = \frac{1}{N}\sum_{j=1}^{N} 1_{K_r \cap \Omega}(\boldsymbol{x}_j \pmod{1}),
	\end{equation*} 
	for any point set $\{\bsx_1\ldots,\bsx_{N}\} \subset \mathbb{R}^2$. Thus we obtain from \eqref{bound_BCGT2} that
	\begin{align}\label{bound_BCGT3}
	 \sup_{\boldsymbol{t}\in [0,1]^2 } \left|\frac{1}{N} \sum_{j=1}^N 1_{[\boldsymbol{0}, \boldsymbol{t}] \cap \Omega}(\bsx_j \pmod{1} )  - \lambda([\boldsymbol{0}, \boldsymbol{t}] \cap \Omega) \right| \le  c'_s N^{-2/3} \log N.
	\end{align}
	
	
	In order to be able apply this result in our setting, it remains to construct a suitable convex set $\Omega$ in $[0,1]^2$ which has the graph of $\psi L^{-1}$ as part of its boundary. We define the boundary of the set $\Omega$ by extending the graph of $\psi L^{-1}$ using a B\'{e}zier curve such that the curve is twice continuously differentiable. The B\'{e}zier curve can be constructed using the derivative information of $\psi L^{-1}$ at the boundary and further control points to control the curvature of the curve. The set $\Omega$ enclosed by this curve then satisfies the assumptions that its boundary is twice continuously differentiable with non-vanishing curvature. The details of the construction are left to the reader (see \cite[Chapter~6]{Farin2001}).
	
	With these settings, the desired discrepancy bound in Theorem~\ref{Th:2thirds} now follows from \eqref{bound_BCGT3}.

	\section{A quality criterion for driver sequences}\label{sec_quality}
	
	In acceptance-rejection sampling, the choice of driver sequence can have a significant  impact on the properties of the accepted samples. In this section, we will present a criterion which can be used to  measure the quality of driver sequences.

	Let $\boldsymbol{n}=(n_1,n_2)\in \mathbb{Z}^2$  and let $|\boldsymbol{n}|=\max\{|n_1|,|n_2|\}$.
	For a collection of points $T_M=\{\boldsymbol{x}_j \mid j=1,\ldots, M\}$ and $R>0$, define the following quantity $\mathcal{Q}_R$ with respect to the point set $T_M$,
	\begin{eqnarray}\label{CriterionR}
	\mathcal{Q}_R(T_M)= \frac{1}{R}+\sum_{\substack{0<|\boldsymbol{n}|<R\\ \boldsymbol{n}\in \mathbb{Z}^2\backslash \boldsymbol{0}}}
	\Big(\frac{1}{|\boldsymbol{n}|^{3/2}}+\frac{1}{(1+|n_1|)(1+|n_2|)}\Big) 
	\cdot \Big| \frac{1}{M}\sum_{j=1}^M e^{2\pi \mathrm{i} \boldsymbol{n} \cdot \bsx_j}\Big|.
	\end{eqnarray}
	
	The general Erd\H{o}s-Tur\'{a}n inequality \cite[Theorem~3]{BCGT2015}  provides an upper bound on the discrepancy.
	We restate  this result as a proposition in the following.
	
	\begin{proposition}\label{Prop:ETineq}
		There exists a positive function $\phi(u)$ on $[0,\infty)$ with rapid decay at infinity such that for every collection of points $\{\bsx_j \mid j=1,\ldots, M \} \subset \mathbb{R}^s$,
		for every bounded Borel set $D \subset  \mathbb{R}^s$, and for every $R>0$,
		
		\begin{eqnarray*}
			&&\Big |\frac{1}{M}\sum_{j=1}^{M}  \sum_{\boldsymbol{m}\in \mathbb{Z}^2} 1_{D} (\bsx_j +\boldsymbol{m})-\lambda(D)\Big|\\
			&\le &|\hat{H}_R(0)|+\sum_{\substack{ \boldsymbol{n}\in \mathbb{Z}^2 \\0<|\boldsymbol{n}|<R }}(|\hat{1}_D(\boldsymbol{n})|+ |\hat{H}_R(n)|) \Big| \frac{1}{N}\sum_{j=1}^M e^{2\pi \mathrm{i} \boldsymbol{n} \cdot \bsx_j}\Big|,
		\end{eqnarray*}
		where $H_R(x)=\phi(R~\mathrm{dist} (x, \partial D))$, where $\mathrm{dist}(x, \partial D)$ is the Euclidean distance between $x$ and the boundary $\partial D$ of $D$, $\hat{H}_R(0)$ is the zeroth Fourier coefficient of $H_R$ and $\hat{1}_{D}$ is the Fourier transform of the indicator function $1_{D}$ of $D$.
	\end{proposition}
	Under certain smoothness conditions on the boundary curve of $D$, the quality criterion $\mathcal{Q}_R$  can be derived from the right-hand side of the Erd\H{o}s-Tur\'{a}n inequality by working out the corresponding Fourier coefficient decay. More precisely, if $D$ is the intersection of a convex set $\Omega$ with a rectangle such that the boundary curve of $\Omega$  is twice continuously differentiable and having non-vanishing curvature, then we have the formula for $\mathcal{Q}_R$  as shown in Equation~\eqref{CriterionR}. This was shown in \cite[Lemma~10 \& Lemma~11]{BCGT2015}. We use this fact in the proof of Theorem~\ref{Lem:2Dis} below.

	The following theorem shows  a connection between the criterion $\mathcal{Q}_R$ for the driver sequence and the star-discrepancy of  the samples obtained by the acceptance-rejection algorithm using a deterministic  driver sequence.  In the following discussion, the notation $ x_N \lesssim y _N$ means that there exists a positive constant $\theta$ such that $x_N\le \theta y_N$ for all $N$.
	
	\begin{theorem}\label{Lem:2Dis}
		Let the  unnormalized density function $\psi:[0,1]\to \mathbb{R}_+$  be twice continuously
		differentiable and having non-vanishing curvature. Assume that there exists a constant $L<\infty$ such that $\psi(x) \leq L$ for  all $x\in[0,1]$. Let $Q_N (T_M; \psi)=\{ y_j \mid j = 1,\ldots, N \} \subset [0,1]$ be generated by the acceptance-rejection sampler using a point set $T_M = \{\bsx_j \mid j = 1, \ldots, M \}$ of cardinality $M$ as the driver sequence. Then we have
		\begin{equation*}
		D_{N,\psi}^*(Q_N(T_M; \psi)) \lesssim \mathcal{Q}_R(T_M).
		\end{equation*}
	\end{theorem}
	
	\begin{proof} 
		Let $$A=\{\boldsymbol{x}=(x_1,x_2)\in[0,1]^2 \mid \psi(x_1)\ge Lx_2\}$$ and $$J_{t}^*=([0,t)\times [0,1))\bigcap A $$ for $t\in [0,1]$.
		
		Algorithm~\ref{DAR} implies that the points $y_1, \ldots, y_N$ are the first coordinates of the points of the driver sequence $\bsx_1, \ldots, \bsx_M$ which are in the set $A$. Hence we have
		\begin{equation*}
		\sum_{j=1}^{M}1_{J_{t}^*}(\boldsymbol{x}_j)
		=\sum_{j=1}^{N}1_{[0,t)}(y_j).
		\end{equation*} 
		Note that $C= \int_0^1 \psi(z) \rd z = L \lambda(A)$ and for any $t \in [0,1]$ we have $\int_0^t \psi(z) \rd z= L \lambda(J_t^*)$.
		Therefore,
		\begin{eqnarray}\label{Discrep-QR}
		&&\Big|\frac{1}{N}\sum_{j=1}^{N}1_{[0,t)}(y_j)
		-\frac{1}{C}\int_{[0,t)}\psi(z)\rd z\Big| \nonumber \\
		&=&\Big|\frac{1}{N}\sum_{j=1}^{M}1_{J_t^*}({\boldsymbol{x}_j})
		- \frac{1}{\lambda(A)} \lambda(J_t^*) \Big| \nonumber \\
		& \le & \frac{M}{N} \Big|\frac{1}{M} \sum_{j=1}^{M} 1_{J_t^*}(\bsx_j) - \lambda(J_t^*) \Big|  + 
		\Big|\lambda(J_t^*) \Big(\frac{M}{N} - \frac{1}{\lambda(A)}\Big)\Big| \nonumber\\
		& \le& \frac{M}{N} \Big(\Big|\frac{1}{M}\sum_{j=1}^{M} 1_{J_t^*}(\bsx_j) - \lambda(J_{t}^*)\Big|  + \Big| \lambda(A) \Big(1- \frac{1}{\lambda(A)} \frac{N}{M}\Big)\Big| \Big) \nonumber\\
		& \le& \frac{M}{N} \Big(\Big|\frac{1}{M}\sum_{j=1}^{M} 1_{J_t^*}(\bsx_j) - \lambda(J_{t}^*)\Big| + \Big| \lambda(A) -  \frac{1}{M}\sum_{j=1}^{M} 1_{J_A}(\bsx_j)\Big| \Big) \nonumber\\
		&\le& \frac{2M}{N}\sup_{t\in [0,1]}\left|\frac{1}{M}\sum_{j=1}^{M} 1_{J_t^*}(\bsx_j) - \lambda(J_{t}^*) \right|,
		\end{eqnarray}
		where we used the estimation $\lambda(J^*_t) \le\lambda(A)$ and the fact that
		$N = \sum_{j=1}^{M} 1_{A}(\bsx_j)$ is the number of accepted points.

		For the Borel set $J_t^*\subseteq[0,1)^2$ we have
		\[\frac{1}{M}\sum_{j=1}^M \sum_{\boldsymbol{m} \in \mathbb{Z}^2} 1_{J_t^*} (\bsx_j + \boldsymbol{m}) = \frac{1}{M}\sum_{j=1}^M 1_{J_t^*}(\boldsymbol{x}_j\pmod{1}).\]
		
		Then by the general Erd\H{o}s-Tur\'{a}n inequality in Proposition~\ref{Prop:ETineq}, we obtain, for every $R>0$,
		\begin{eqnarray*}
			\Big|\frac{1}{M}\sum_{j=1}^{M}1_{J_t^*}(\bsx_j)
			-\lambda(J_{t}^*)\Big| \le |\hat{H}_R(0)|
			+\sum_{\substack{ \boldsymbol{n}\in \mathbb{Z}^2\\0<|\boldsymbol{n}|<R}}(|\hat{1}_{J_t^*}(\boldsymbol{n})|+ |\hat{H}_R(\boldsymbol{n})|) \Big| \frac{1}{M}\sum_{j=1}^M e^{2\pi \mathrm{i} \boldsymbol{n} \cdot \bsx_j}\Big|.
		\end{eqnarray*}

		Note that $J_t^*$ is the intersection of the convex set $\Omega$  whose boundary was constructed using a  B\'{e}zier curve  which  is twice continuously differentiable with non-vanishing curvature (see the proof of Theorem~\ref{Th:2thirds}), and the rectangle $[0,t)\times [0,1)$. Thus we can use the following estimations from   \cite[Lemma~10 \& Lemma~11]{BCGT2015},
		\begin{eqnarray*}
			|\hat{H}_R(0)|&\lesssim &\frac{1}{R},\\
			|\hat{H}_R(\boldsymbol{n})| &\lesssim & \frac{1}{|\boldsymbol{n}|^{3/2}}+ \frac{1}{(1+|n_1|)(1+|n_2|)},\\
			|\hat{1}_{J_t^*}(\boldsymbol{n})|&\lesssim& \frac{1}{|\boldsymbol{n}|^{3/2}}+ \frac{1}{(1+|n_1|)(1+|n_2|)},\\
		\end{eqnarray*}
		the result now follows.
	\end{proof}

	\section{Proof of Theorem~\ref{Th:lattice}}\label{sec:PT2}
	
	We prove the following lemma which, together with Theorem~\ref{Lem:2Dis}, implies Theorem~\ref{Th:lattice}.
	
	\begin{lemma}\label{Bound:QR} Let $F_k$ denote the $k$-th Fibonacci number.
		Let $$\mathcal{F}_k= \Big\{\bsx_j=\big(\frac{j}{F_k}, \big\{\frac{jF_{k-1}}{F_k}\big\} \big) \mid j=1,\ldots, F_k\Big\}.$$ Then we have
		\begin{equation*}
		\mathcal{Q}_R (\mathcal{F}_k) \lesssim F_{k}^{-2/3},
		\end{equation*}  
		for $R=F_{\lceil 2k/3\rceil}$. The implied  constant is independent of $F_k$.
	\end{lemma}
	
	\begin{proof}
		The quantity $\mathcal{Q}_R$ was defined in \eqref{CriterionR}.  We first consider the right-most sum in the definition of $\mathcal{Q}_R$, for which we have
		\begin{eqnarray*}
			\Big|\frac{1}{F_k}\sum\limits_{j=1}^{F_k} e^{2\pi \mathrm{i} \boldsymbol{n} \cdot \big(\frac{j}{F_k},\frac{jF_{k-1}}{F_k}\big)}\Big| &=&\Big|\frac{1}{F_k}\sum\limits_{j=1}^{F_k} \big(e^{2\pi \mathrm{i} \boldsymbol{n} \cdot (1,F_{k-1})/F_k}\big)^j\Big|\\ &=&\left\{\begin{array}{ll}
				1, & \mbox{if~} F_k|(n_1+n_2F_{k-1}),\\[0.25cm]
				0, &\mbox{if~} F_k\nmid(n_1+n_2F_{k-1}),\\[0.25cm]
			\end{array}\right.
		\end{eqnarray*}
		where  $F_k|(n_1+n_2F_{k-1})$ means that  $F_k$ divides $(n_1+n_2F_{k-1})$. Hence
		\begin{eqnarray}\label{SumQR}
		&&\sum\limits_{\substack{\boldsymbol{n}\in \mathbb{Z}^2\\0<|\boldsymbol{n}|<R}} \Big(\frac{1}{|\boldsymbol{n}|^{3/2}}+\frac{1}{(1+|n_1|)(1+|n_2|)}\Big)
	 \cdot \Big|\frac{1}{F_k}\sum\limits_{j=1}^{F_k} e^{2\pi \mathrm{i} \boldsymbol{n} \cdot \big(\frac{j}{F_k},\frac{jF_{k-1}}{F_k}\big)}\Big| \nonumber \\
		=&&\sum\limits_{\substack{\boldsymbol{n}\in \mathbb{Z}^2\\ 0<|\boldsymbol{n}|<R\\   F_k|(n_1+n_2F_{k-1})}} \Big(\frac{1}{|\boldsymbol{n}|^{3/2}}+\frac{1}{(1+|n_1|)(1+|n_2|)}\Big). \nonumber
		\end{eqnarray}

		From \cite[Definition~5.4 \& Equation~(5.11) \& Theorem~5.17]{Niederreiter1992} we obtain that
		\begin{align*}
		\sum_{\boldsymbol{n} } \frac{1}{\max\{1,|n_1|\}\max\{1,|n_2|\}}  \lesssim & \frac{(\log F_k)^2}{F_k},
		\end{align*}
		where the sum is over all $\boldsymbol{n} = (n_1, n_2) \neq (0, 0)$ with $n_1 + n_2 F_{k-1} \equiv 0 \pmod{F_k}$ and $ - F_k/2 < n_i \le F_k/2$ for $i = 1, 2$. Hence
		\begin{eqnarray*}
			\sum\limits_{\substack{0<|\boldsymbol{n}|<R\\ F_k|(n_1+n_2F_{k-1})}} \frac{1}{(1+|n_1|)(1+|n_2|)}
			&\le& \sum\limits_{\substack{0<|\boldsymbol{n}|<R\\F_k|(n_1+n_2F_{k-1})}} \frac{1}{\max\{1,|n_1|\}\max\{1,|n_2|\}}\\
			 & \lesssim &  \frac{(\log F_k)^2}{F_k}.
		\end{eqnarray*}
		
		Note that if $n_1=0$, then $F_k|n_2F_{k-1}$ which implies $F_k|n_2$ since $\gcd(F_k,F_{k-1})=1$. It further
		implies $n_2=0$ by realising that $|n_2|\le R =F_{\lceil 2k/3\rceil} < F_k$ for $k \ge 3$. 
		
		If $n_2=0$, then $F_k | n_1$, which implies that $n_1 = 0$ since $|n_1| \le R = F_{\lceil 2k/3 \rceil} < F_k$ for $k \ge 3$.  
		
		Since $F_k | (n_1 + n_2 F_{k-1})$, there is an $\ell \in \mathbb{Z}$ such that $n_1+n_2F_{k-1}=\ell F_k$. 
		For given $n_2$, there is at most one value $\ell\in \mathbb{Z}$ such that $-R<n_1=\ell F_k-n_2F_{k-1}< R$.
		
		Now we estimate the remaining term of Equation~\eqref{SumQR}. We have
		\begin{eqnarray}\label{termN}
		\sum\limits_{\substack{0<|\boldsymbol{n}|<R\\ F_k|(n_1+n_2F_{k-1})}} \frac{1}{|\boldsymbol{n}|^{3/2}}
		=\sum\limits_{\substack{-R<n_2<R\\ n_2\neq 0}}\sum\limits_{\substack{\ell\in \mathbb{Z}\\ -R<\ell F_k-n_2F_{k-1}<R}}
		\frac{1}{\max\{|n_2|,|\ell F_k-n_2F_{k-1}|\}^{\frac{3}{2}}}.
		\end{eqnarray}
		
		To bound this term we need some preliminary results on Fibonacci lattice point sets $\mathcal{F}_k$. The star-discrepancy with respect to uniform distribution (given in Equation~\eqref{defStarDis} using $\psi=1$) of the Fibonacci point set $\mathcal{F}_k$ is bounded by 
		$$D_{F_k}^*(\mathcal{F}_k)\le c_0 \frac{\log F_k}{F_k},$$
		see  \cite[pp.~124]{Niederreiter1992}. The star-discrepancy  $D_{F_k}^*(\mathcal{F}_k)$ is defined with respect to   rectangles $[\boldsymbol{0},\boldsymbol{t}) = [0,t_1)\times[0,t_2)$ for all $(t_1,t_2)\in[0,1]^2$. To switch to the discrepancy $D_{F_k}(\mathcal{F}_k)$ with respect to arbitrary rectangles $[\boldsymbol{a},\boldsymbol{b}) \subseteq [0,1]^2$, we use the inequality $D_{F_k}(\mathcal{F}_k)\le 4D_{F_k}^*(\mathcal{F}_k)$, see \cite[Proposition~2.4]{Niederreiter1992}.

		Consider a rectangle $V$ of the following form,
		\begin{eqnarray*}
			V=\Big[ \frac{a}{F_k}, \frac{a+u}{F_k}\Big) \times \Big[ \frac{b}{F_k}, \frac{b+v}{F_k}\Big).
		\end{eqnarray*}
		By the definition of the star-discrepancy, it follows that
		\begin{eqnarray*}
			\Big| \frac{|V\cap \mathcal{F}_k|}{F_k} - \frac{uv}{F_k^2}\Big|	\le D_{F_k}(\mathcal{F}_k)\le 4D_{F_k}^*(\mathcal{F}_k)\le 4c_0 \frac{\log F_k}{F_k}.
		\end{eqnarray*}
		This implies that
		\begin{eqnarray}\label{FiboDiscrep}
		|V\cap \mathcal{F}_k| \le 4c_0 \log F_k + \frac{uv}{F_k}.
		\end{eqnarray}

		We now consider the double sum in Equation~\eqref{termN}. We divide the range of $0<|n_2|<F_{\lceil 2k/3\rceil}$ into
		\begin{equation*}
		F_i\le |n_2|< F_{i+1} \mbox{ for } i=2,3,\ldots, \big\lceil \frac{2k}{3}\big\rceil-1.
		\end{equation*}
		Let $a=F_i$ and $u=F_{i-1}$, then $a+u=F_{i+1}$.
		Similarly we divide the range of $0<|n_1|=|\ell F_k-n_2F_{k-1}|<F_{\lceil 2k/3\rceil}$
		into
		\begin{equation*}
		F_m\le|n_1| < F_{m+1} \mbox{ for } m = 2,3,\ldots, \big\lceil \frac{2k}{3}\big\rceil-1.
		\end{equation*}
		Let $b=F_m$ and $v=F_{m-1}$, then $b+v=F_{m+1}$. With those settings we have
		\begin{eqnarray*}
			\begin{array}{ll}
				&\displaystyle \frac{a}{F_k}\le \frac{|n_2|}{F_k} < \frac{a+u}{F_k},        \\[0.25cm]
				&\displaystyle \frac{b}{F_k} \le\big| \frac{n_2F_{k-1}}{F_k}-\ell\big| < \frac{b+v}{F_k}.     \\[0.25cm]
			\end{array}
		\end{eqnarray*}
		
		By Equation~\eqref{FiboDiscrep}, the number of Fibonacci points in the rectangle $V$,  given by $|\mathcal{F}_k\cap V|$,  is therefore bounded by $4c_0\log F_k+\frac{F_{i-1}F_{m-1}}{F_k}$. This is equivalent to
		the statement that the number of $(n_1,n_2)$ with $n_1=\ell F_k-n_2F_{k-1}$, $a=F_i\le |n_2|<F_{i+1}=a+u$, and  $b=F_m\le |n_1|<F_{m+1}=b+v$ is
		bounded above by a constant (which is independent of $i, k, m$) times
		
		\begin{equation}\label{NumBound}
		\quad 4c_0\log F_k+\frac{F_{i-1}F_{m-1}}{F_k}.
		\end{equation}
		
		This result can  be obtained by considering the following four cases,
		\begin{itemize}
			\item [(i)] $a\le n_2 < a+u$ and $b\le n_2F_{k-1}-\ell F_k<b+v$,
			\item [(ii)] $a\le -n_2 < a+u$ and $b\le -(n_2F_{k-1}-\ell F_k)<b+v$,
			\item [(iii)] $a\le n_2 < a+u$ and $b\le -(n_2F_{k-1}-\ell F_k)<b+v$,
			\item [(iv)] $a\le -n_2 < a+u$ and $b\le n_2F_{k-1}-\ell F_k<b+v$.
		\end{itemize}
		More precisely, for case (iii) and (iv), we consider the point set
		\begin{eqnarray*}
			\mathcal{F}'_k
			&=& \Big\{  \big(\frac{j}{F_k}, \big\{-\frac{jF_{k-1}}{F_k}\big\} \big) \mid j=1,\ldots F_k\Big\}\\
			&=& \Big\{  \big(\frac{j}{F_k}, 1-\big\{\frac{jF_{k-1}}{F_k}\big\} \big)\mid j=1,\ldots F_k\Big\}.
		\end{eqnarray*}
		Then the star-discrepancy of $\mathcal{F}'_k$, $D_{F_k}^*(\mathcal{F}'_k)\lesssim \frac{\log F_k}{F_k}$ by noting that  $\bsx_j'$ is a reflection
		of $\bsx_j \in \mathcal{F}_k$ and the inequalities $D_{F_k}^*(\mathcal{F}_k)\le D_{F_k}(\mathcal{F}'_k)\le 4D_{F_k}^*(\mathcal{F}'_k)$.
		
		On the other hand, for all $1\le n_2 < F_{i+1}$ and $k>i$, using the continued fractions technique mentioned in \cite[Appendix~B]{Niederreiter1992}
		and a property of Fibonacci numbers, we obtain
		\begin{equation}\label{n1Bound}
		|n_1|=|\ell F_k-n_2F_{k-1}| \ge |F_{i-1}F_k - F_iF_{k-1}|=F_{k-i}.
		\end{equation}
		Since  $|n_1|\le |\boldsymbol{n}|<R=F_{\lceil2k/3 \rceil}$, there is no solution if $k-i\ge \big\lceil\frac{2k}{3}\big\rceil$, i.e. $i\le \lfloor \frac{k}{3}\rfloor$.
		Therefore
		\begin{eqnarray} \label{sumQR1}
		&&\sum\limits_{\substack{0<|\boldsymbol{n}|<R\\ \boldsymbol{n}\in \mathbb{Z}^2}} \frac{1}{|\boldsymbol{n}|^{3/2}}\nonumber\\
		&=&\sum\limits_{\substack { -R<n_2<R \\n_2\neq 0}}\sum\limits_{\substack{\ell\in \mathbb{Z}\\ -R<\ell F_k-n_2F_{k-1}<R}}
		\frac{1}{\max\{|n_2|,|\ell F_k-n_2F_{k-1}|\}^{3/2}} \nonumber \\
		&=&2 \sum\limits_{i=\lfloor k/3\rfloor+1}^{\lceil 2k/3\rceil-1}\sum\limits_{n_2=F_i}^{F_{i+1}-1}
		\sum\limits_{m=k-i}^{ \lfloor 2k/3\rfloor -1 } 
	\sum\limits_{\substack {\ell \in\mathbb{Z}\\ F_m\le|\ell F_k-n_2F_{k-1}|<F_{m+1}} }
		\frac{1}{(\max\{|n_2|,|\ell F_k-n_2F_{k-1}|\})^{3/2}}, \nonumber\\
		&&
		\end{eqnarray}
		where we used that for $(n_1,n_2)\in \mathbb{Z}^2$ which satisfy $0<\max \{|n_1|,|n_2|\}<R$ and $F_k|(n_1+n_2F_{k-1})$,
		also $(-n_1,-n_2)\in\mathbb{Z}^2$ satisfy these properties.
		
		To further estimate the right-hand side of Equation~\eqref{sumQR1}, we use the following inequalities. For $F_i\le |n_2|<F_{i+1}$ and $F_m\le |n_1|<F_{m+1}$ we have
		\begin{equation*}
		\max\{|n_2|,|\ell F_k-n_2F_{k-1}|\} \ge \max\{F_i, F_{m}\}.
		\end{equation*}
		For $\lfloor \frac{k}{3}\rfloor< i< \lceil \frac{k}{2}\rceil $ and $k-i\le m<k$,
		we have $ \max\{F_i, F_{m}\}=F_m$.

		Applying the bound given in \eqref{NumBound} in each case, we obtain from Equation~\eqref{sumQR1} that
		\begin{eqnarray*}
			\sum\limits_{\substack{0<|\boldsymbol{n}|<R\\ \boldsymbol{n}\in \mathbb{Z}^2}} \frac{1}{|\boldsymbol{n}|^{3/2}} 
			&\le& 2\sum\limits_{i=\lfloor k/3\rfloor+1}^{\lceil k/2\rceil-1}\sum\limits_{m=k-i}^{k} \Big(\frac{4c_0\log F_k}{F_m^{3/2}}+\frac{F_{i-1}F_{m-1}}{F_m^{3/2}F_k}\Big) \nonumber \\
			&&+2\sum\limits_{i=\lceil k/2\rceil}^{\lceil 2k/3\rceil -1}\sum\limits_{m=k-i}^{ \lfloor 2k/3\rfloor -1} \Big( \frac{4c_0\log F_k}{\max\{F_i, F_m\}^{\frac{3}{2}}}+\frac{F_{i-1}F_{m-1}}{\max\{F_i, F_m\}^{\frac{3}{2}}F_k}\Big).
		\end{eqnarray*}
		
		It is well known that  $F_i=[\varphi^i/\sqrt{5}]$ with $\varphi=(1+\sqrt{5})/2$, where $[\cdot]$ denotes  the nearest integer function given by the integer $[x]=\gamma \in \mathbb{Z}$ which satisfies that $\gamma-\frac{1}{2} < x \le \gamma +\frac{1}{2}$. Thus we have
		\begin{eqnarray}\label{sum1-1}
		\sum\limits_{i=\lfloor k/3\rfloor+1}^{\lceil k/2\rceil-1}\sum\limits_{m=k-i}^{ \lfloor 2k/3\rfloor -1} \frac{\log F_k}{F_m^{3/2}} 
		&\lesssim&\log F_k \sum\limits_{i=\lfloor k/3\rfloor +1}^{\lceil k/2\rceil -1}\sum\limits_{m=k-i}^{ \lfloor 2k/3\rfloor -1} \frac {1} {\varphi^{\frac{3}{2}m}} \nonumber\\
		&\lesssim&\log F_k \sum\limits_{i=\lfloor k/3 \rfloor+1}^{\lceil k/2\rceil-1} \frac{\varphi^{-\frac{3}{2}(k-i)} (1- \varphi^{-\frac{3}{2}(i-k/3+1)})}{1-\varphi^{-3/2}}  \nonumber\\
		&\lesssim&\frac{\log F_k}{\varphi^{\frac{3}{2}k}}\sum\limits_{i=\lfloor k/3\rfloor+1}^{\lceil k/2\rceil-1} \varphi^{\frac{3}{2}i} \nonumber \\
		&\lesssim&\frac{\log F_k}{\varphi^{\frac{3}{2}k}}\varphi^{\frac{3}{4}k}\lesssim \frac{\log F_k}{F_k^{3/4}},
		\end{eqnarray}
		and
		\begin{eqnarray}\label{sum1-2}
	\sum\limits_{i=\lfloor k/3\rfloor+1}^{\lceil k/2\rceil-1}\sum\limits_{m=k-i}^{ \lfloor 2k/3\rfloor -1} \frac{F_{i-1}F_{m-1}}{F_m^{3/2}F_k}  &\lesssim&\sum\limits_{i=\lfloor k/3\rfloor+1}^{\lceil k/2\rceil-1}\sum\limits_{m=k-i}^{ \lfloor 2k/3\rfloor -1} \varphi^{i-k-\frac{1}{2}m} \nonumber \\
		&\lesssim & \sum\limits_{i=\lfloor k/3\rfloor+1}^{\lceil k/2\rceil -1}\varphi^{i-k}
		\frac{\varphi^{-\frac{1}{2}(k-i)}(1-\varphi^{-\frac{1}{2}(i-k/3+1)} )} {1-\varphi^{-1/2}} \nonumber \\
		&\lesssim & \frac{1}{\varphi^{\frac{3}{2}k} } \sum\limits_{i=\lfloor k/3\rfloor +1}^{\lceil k/2\rceil -1}\varphi^{\frac{3}{2}i}
		\lesssim \frac{1}{\varphi^{\frac{3}{4}k}}\lesssim \frac{1}{ F_k^{3/4}}.
		\end{eqnarray}
		With respect to the second summation,
		\begin{eqnarray}\label{sum2-1}
		\sum\limits_{i=\lceil k/2\rceil }^{\lceil 2k/3\rceil -1}\sum\limits_{m=k-i}^{ \lfloor 2k/3\rfloor -1}  \frac{\log F_k}{\max\{F_i,F_m\}^{3/2}}
		&\lesssim& \log F_k \sum\limits_{i=\lceil k/2\rceil }^{\lceil 2k/3\rceil -1}\Big(\sum\limits_{m=k-i}^{i}\varphi^{-\frac{3}{2}i}+\sum\limits_{m=i+1}^{ \lfloor 2k/3\rfloor -1} \varphi^{-\frac{3}{2}m}\Big)\nonumber \\
		&\lesssim& \log F_k \sum\limits_{i=\lceil k/2\rceil }^{\lceil 2k/3\rceil -1}k\varphi^{-\frac{3}{2}i}
		\lesssim \frac{(\log F_k )^2}{\varphi^{\frac{3}{4}k}} \lesssim \frac{(\log F_k )^2} {F_k^{3/4}}.\nonumber \\
		&&
		\end{eqnarray}
		Moreover, we obtain
		\begin{eqnarray}\label{sum2-2}
		\sum\limits_{i=\lceil k/2\rceil}^{\lceil 2k/3\rceil-1}\sum\limits_{m=k-i}^{ \lfloor 2k/3\rfloor -1} \frac{F_{i-1}F_{m-1}}{\max\{F_i, F_m\}^{3/2}F_k} 
		&\lesssim& \sum\limits_{i=\lceil k/2\rceil }^{\lceil 2k/3\rceil-1}\sum\limits_{m=k-i}^{ \lfloor 2k/3\rfloor -1} \varphi^{i+m-k-\frac{3}{2} \max\{i,m\}} \nonumber \\
		&=&\sum\limits_{i=\lceil k/2\rceil}^{\lceil2k/3\rceil-1}\Big(\sum\limits_{m=k-i}^{i} \varphi^{m-k-\frac{1}{2}i}+\sum\limits_{m=i+1}^{  \lfloor 2k/3\rfloor -1} \varphi^{i-k-\frac{1}{2}m}\Big) \nonumber \\
		&\lesssim& \sum\limits_{i=\lceil k/2\rceil}^{\lceil 2k/3\rceil-1}\varphi^{\frac{1}{2}i-k} + \sum\limits_{i=\lceil k/2\rceil}^{\lceil 2k/3\rceil-1}\varphi^{\frac{1}{2}i-k}\nonumber \\
		&\lesssim& \varphi^{-\frac{2}{3}k} + \varphi^{-\frac{2}{3}k}\lesssim \frac{1}{F_k^{2/3}}.
		\end{eqnarray}
		
		Since $\frac{(\log F_k)^2}{F_k^{3/4}}$ converges faster to $0$ than $F_k^{-2/3}$, we obtain a convergence rate of order $F_k^{-2/3}$ of the right-hand side of \eqref{SumQR}. By setting $R=F_{\lceil 2k/3\rceil}$ we obtain a convergence rate of order $F_k^{-2/3}$ for $\mathcal{Q}_R (\mathcal{F}_k)$,
		which completes the proof.
	\end{proof}
	\begin{remark} Note that choosing $R$ differently does not improve our result. Since Equation~\eqref{CriterionR} contains the factor $\frac{1}{R}$, we need to choose $F_k^{2/3}\lesssim R$. Choosing $R$ larger than that can only increase the second term in \eqref{CriterionR}. But for this second term we proved a convergence of order $F_k^{-2/3}$ for $R$ of order $F_k^{2/3}$. Hence we cannot improve our result using a larger value of $R$.
	\end{remark}
	
	\begin{remark}
		Lattice point sets of the form $\{(\frac{j}{N},\{\frac{g}{N}\}),j=1,2,\ldots,N\}$ have small star-discrepancy with respect to rectangular boxes if the coefficients in the  continued fraction  expansion of $\frac{g}{N}$ are bounded independently of $N$, see \cite[Theorem~5.17]{Niederreiter1992}. In particular, for Fibonacci lattice point sets these coefficients are always $1$.  Niederreiter \cite{Niederreiter1986} explicitly finds values of $g$ for $N$ of the form $2^\ell$, $3^\ell$, $5^\ell$, such that the continued fraction coefficient are at most $3$ for $2^\ell$ and $3^\ell$, and at most $4$ for $5^\ell$. It is reasonable to suggest that similar results to Theorem~\ref{Th:lattice} and Corollary~\ref{Corr:integErr}  can also be obtained for lattice point sets based on the results in \cite{Niederreiter1986}.
	\end{remark}

	\section*{Acknowledgements}
The work was supported
		under the Australian Research Council's Discovery Projects funding scheme (project DP150101770).
	The authors are grateful to the anonymous referees and handling editor for valuable comments which improved the presentation of the results.

\end{document}